\documentclass[11pt]{amsart}
\usepackage{amssymb}
\usepackage[margin=1in]{geometry}
\usepackage[colorlinks]{hyperref}
\newtheorem{theorem}{Theorem}[section]

\newtheorem{prop}[theorem]{Proposition}
\newtheorem{cor}[theorem]{Corollary}

\theoremstyle{definition}
\newtheorem{definition}[theorem]{Definition}

\theoremstyle{remark}
\newtheorem{remark}[theorem]{Remark}

\numberwithin{equation}{section}
\begin{document}

\newcommand{\spacing}[1]{\renewcommand{\baselinestretch}{#1}\large\normalsize}
\spacing{1.14}

\title{Randers Ricci Soliton Homogeneous Nilmanifolds}

\author {H. R. Salimi Moghaddam}

\address{Department of Pure Mathematics \\ Faculty of  Mathematics and Statistics\\ University of Isfahan\\ Isfahan\\ 81746-73441-Iran.} \email{hr.salimi@sci.ui.ac.ir and salimi.moghaddam@gmail.com}


%%\date{\today}

\begin{abstract}
    Let $F$ be a left invariant Randers metric on a simply connected nilpotent Lie
    group $N$, induced by a left invariant Riemannian metric ${\hat{\textbf{\textit{a}}}}$ and a vector field $X$ which is $I_{\hat{\textbf{\textit{a}}}}(M)$-invariant. If the Ricci flow equation has a unique solution then, $(N,F)$ is a Ricci soliton if and only if $(N,F)$ is a semialgebraic Ricci soliton.
\end{abstract}

\maketitle
%%---------------------------INTRODUCTION--------------------------

\section{\textbf{Introduction}}
In 1982, R. S. Hamilton has shown that any compact Riemannian $3-$manifold with strictly positive Ricci curvature, admits a Riemannian metric of
constant positive curvature (see \cite{Hamilton}). He has started with a Riemannian manifold $(M,{\hat{\textbf{\textit{a}}}}_0)$ and evolved its metric by the following geometric evolution equation which is named un-normalized Ricci flow
\begin{equation}\label{un-normalized Ricci flow equation}
    \frac{\partial}{\partial t}{\hat{\textbf{\textit{a}}}}_{ij}=-2ric_{ij},
\end{equation}
where $ric_{ij}$ denotes the Ricci curvature tensor of the Riemannian metric ${\hat{\textbf{\textit{a}}}}$. For a compact Riemannian manifold, one can consider the following equation
\begin{equation}\label{normalized Ricci flow equation}
    \frac{\partial}{\partial t}{\hat{\textbf{\textit{a}}}}_{ij}=\frac{2}{n}sc({\hat{\textbf{\textit{a}}}}){\hat{\textbf{\textit{a}}}}_{ij}-2ric_{ij},
\end{equation}
where $\dim M=n$ and $sc({\hat{\textbf{\textit{a}}}})$ is the average of scalar curvature (\cite{Hamilton}). The last equation is called the normalized Ricci flow which is different from un-normalized Ricci flow \ref{un-normalized Ricci flow equation} by a change of scale in space and a change of parametrization in time. In this case the volume of the solution is constant with respect to time (for more details see \cite{Hamilton}). A Riemannian metric is a fixed point of \ref{un-normalized Ricci flow equation} if and only if it is Ricci flat. For a compact manifold, a Riemannian metric is a fixed point of \ref{normalized Ricci flow equation} if and only if it is Einstein (see \cite{Chow-Knopf}). There is a larger class of solutions, called self-similar solutions or Ricci solitons, which is considered as a generalization of fixed points. Let $(M,{\hat{\textbf{\textit{a}}}}(t))$ be a solution of \ref{un-normalized Ricci flow equation} on a time interval $(t_1,t_2)$ containing zero with ${\hat{\textbf{\textit{a}}}}_0={\hat{\textbf{\textit{a}}}}(0)$. Then ${\hat{\textbf{\textit{a}}}}(t)$ is a self-similar solution if there exist scalars $\tau(t)$ and diffeomorphisms $\phi_t$ such that ${\hat{\textbf{\textit{a}}}}(t)=\tau(t)\phi_t^\ast{\hat{\textbf{\textit{a}}}}_0$. On the other hand a Riemannian manifold $(M,{\hat{\textbf{\textit{a}}}}_0)$ is named a Ricci soliton if
\begin{equation}\label{Ricci soliton equation}
    \mathcal{L}_X{\hat{\textbf{\textit{a}}}}_0=2(\lambda {\hat{\textbf{\textit{a}}}}_0-ric({\hat{\textbf{\textit{a}}}}_0)),
\end{equation}
where $\mathcal{L}_X{\hat{\textbf{\textit{a}}}}$ denotes the Lie derivative of ${\hat{\textbf{\textit{a}}}}$ with respect to a vector field $X$. A Ricci soliton is called shrinking if $\lambda>0$, steady if $\lambda=0$ and expanding if $\lambda<0$. It can be shown that there exists a bijection between the family of self-similar solutions and the family of Ricci solitons (see lemma 2.4 of \cite{Chow-Knopf}). This fact allows us to consider them as equivalent.

D. Bao, in \cite{Bao}, has used Akbar-Zadeh's Ricci tensor \cite{Akbar-Zadeh}, to generalize the Ricci flow to Finsler geometry which is a natural extension of Riemannian geometry. In \cite{Bidabad-Yarahmadi}, the concept of Ricci soliton is developed to Finsler geometry by B. Bidabad and M. Yarahmadi.

In 1941 an important family of Finsler metrics, which are called Randers metrics now, introduced by G. Randers for its application in general relativity (see \cite{Randers}). Also, they are the only Finsler metrics which are the solutions of the navigation problem on a Riemannian manifold. In fact the perturbation of a Riemannian metric ${\hat{\textbf{\textit{a}}}}$ by a vector field $W$, with $\Vert W\Vert_{\hat{\textbf{\textit{a}}}} < 1$, generates a Randers metric and conversely every Randers metric can be accomplished through the perturbation of a Riemannian metric ${\hat{\textbf{\textit{a}}}}$ by a vector field $W$ satisfying $\Vert W\Vert_{\hat{\textbf{\textit{a}}}} < 1$ (see \cite{Bao-Robles-Shen}). The importance of Randers spaces, among Finsler manifolds, leads us to study Ricci solitons on them. In this work, we study the conditions under which a left invariant Randers metric on a simply connected nilpotent Lie group $N$, is a Ricci soliton.

In \cite{Lauret}, J. Lauret has shown that a Riemannian nilmanifold  $(N,{\hat{\textbf{\textit{a}}}})$ is a Ricci soliton if and only if it's $(1,1)$ Ricci tensor satisfies the equation
\begin{equation}\label{Riemannian semialgebraic Ricci soliton equation}
    Ric_{\hat{\textbf{\textit{a}}}}=c\mathrm{Id}+\frac{1}{2}(\mathrm{D}+\mathrm{D}^t),
\end{equation}
for some $c\in\mathbb{R}$ and some $\mathrm{D}\in Der(\frak{n})$, where $N$ is a nilpotent Lie group with Lie algebra $\frak{n}$, ${\hat{\textbf{\textit{a}}}}$ is a left invariant Riemannian metric on $N$ and $\mathrm{D}^t$ is the transpose of $\mathrm{D}$ with respect to ${\hat{\textbf{\textit{a}}}}$. In the general case, a homogeneous Riemannian manifold is called semi-algebraic if it satisfies the homogeneous space version of the above equation (for more details see \cite{Jablonski1}). In this case if $\mathrm{D}$ is symmetric then the homogeneous Riemannian manifold is called algebraic Ricci soliton. Recently, M. Jablonski has proved that every homogeneous Riemannian Ricci soliton must be semi-algebraic \cite{Jablonski2}. He has completed his result in his next paper \cite{Jablonski1}, by showing that every semi-algebraic Ricci soliton is algebraic. This means that homogeneous Riemannian Ricci solitons are algebraic.

We say that a vector field $X$ on a Riemannian manifold $(M,{\hat{\textbf{\textit{a}}}})$ is $I_{\hat{\textbf{\textit{a}}}}(M)$-invariant if for any $\phi\in I_{\hat{\textbf{\textit{a}}}}(M)$ and $x\in M$ we have $\phi_\ast (X_x)=X_{\phi(x)}$. In this article we generalize the definition of semi-algebraic Ricci soliton to Finsler spaces. We show that for a left invariant Randers metric $F$ induced by a left invariant Riemannian metric ${\hat{\textbf{\textit{a}}}}$ and an $I_{\hat{\textbf{\textit{a}}}}(M)$-invariant vector field $X$, on a simply connected nilpotent Lie group $N$, the Finslerian nilmanifold $(N,F)$ is a Ricci soliton if and only if $(N,F)$ is a semialgebraic Ricci soliton. This is a generalization of Lauret's result in \cite{Lauret} to Randers spaces.

%%---------------------------Preleminiaries--------------------------

\section{\textbf{Preliminaries}}
In this section we review some fundamental definitions and concepts of Finsler geometry. For more details we refer the readers to \cite{Bao-Chern-Shen} and \cite{Deng-Book}.
\begin{definition}
A Finsler manifold $(M,F)$ is a smooth manifold $M$ equipped with a map (which is called a Finsler metric) $F:TM \longrightarrow \left[ 0, \infty \right)$ such that
\begin{itemize}
\item[i)] $F$ be a smooth function on $TM\backslash \{0\}$,
\item[ii)] $F(x,\lambda y)=\lambda F(x,y)$, \ \ \ \ for all $\quad  \lambda>0$,
\item[iii)] The hessian matrix $(g_{ij})=\left( \dfrac{1}{2} \dfrac{\partial ^2 F^2}{\partial y^i \partial y^j}\right) $ be positive definite for all $(x,y) \in TM\backslash \{0\}$.
\end{itemize}
\end{definition}
An important special type of Finsler metrics is Randers metric. This metric is of the form $F(x,y)=\sqrt{{\hat{\textbf{\textit{a}}}}\left( y ,y \right)}+\beta(x,y)$, where ${\hat{\textbf{\textit{a}}}}\left( y ,y \right)={\hat{\textbf{\textit{a}}}}_{ij}y^iy^j$ and  $\beta(x,y)=b_iy^i$ are a Riemannian metric and a $1-$form on $M$, respectively, such that $\Vert \beta _x\Vert _{\alpha}=\sqrt{{\hat{\textbf{\textit{a}}}}^{ij} (x)b_i(x) b_j(x) } < 1$ (\cite{Chern-Shen}). Using the induced inner product by the Riemannian metric ${\hat{\textbf{\textit{a}}}}$ on any cotangent space $T^*_xM$, the 1-form $\beta$ corresponds to a vector field $X$ on $M$ such that
\begin{equation}
{\hat{\textbf{\textit{a}}}}(y,X(x)) = \beta (x,y).
\end{equation}
Therefore we can consider the Randers metrics as follows
\begin{equation}\label{Randers}
F(x,y)=\sqrt{{\hat{\textbf{\textit{a}}}}\left( y ,y \right)}+{\hat{\textbf{\textit{a}}}}\left( X\left( x\right) ,y\right).
\end{equation}

For any Finsler manifold $(M,F)$ we define
\begin{equation}\label{Cartan coefficients}
    C_{ijk}:=\frac{1}{4}(F^2)_{y^iy^jy^k}.
\end{equation}

Suppose that $(M, F)$ is an arbitrary Finsler manifold and $(x^1, x^2, \cdots , x^n)$ is a local coordinate system on an open
subset $U$ of $M$. The fundamental tensor and the Cartan tensor, are defined as follows, respectively
\begin{eqnarray}
  g_{(x,y)}(u,v) &=& g_{ij}(x,y)u^iv^j \label{Fundamental tensor}\\
  C_{(x,y)}(u,v,w) &=& C_{ijk}(x,y)u^iv^jw^k \label{Cartan tensor},
\end{eqnarray}
for any $(x,y)\in TM\backslash \{0\}$, where $u=u^i\frac{\partial}{\partial x^i}|_x$, $v=v^i\frac{\partial}{\partial x^i}|_x$ and $w=w^i\frac{\partial}{\partial x^i}|_x$. We define $C^i_{jk}:=g^{is}C_{sjk}$, where $(g^{ij})$ denotes the inverse of the matrix $(g_{ij})$. Now we can have the Christoffel symbols of the second kind, on $TU\backslash \{0\}$, as follows
\begin{equation}\label{Christoffel symbols of the second kind}
    \gamma^i_{jk}:=\frac{1}{2}g^{is}(\frac{\partial g_{sj}}{\partial x^k}-\frac{\partial g_{jk}}{\partial x^s}+\frac{\partial g_{ks}}{\partial x^j}).
\end{equation}
We also define the nonlinear connection $N^i_j$ by the following equation
\begin{equation}\label{nonlinear connection}
    N^i_j:=\gamma^i_{jk}y^k-C^i_{jk}\gamma^k_{rs}y^ry^s,
\end{equation}
where $y=y^i\frac{\partial}{\partial x^i}$.
Now we have the following theorem.
\begin{theorem}\label{Chern connection}
(see \cite{Bao-Chern-Shen})
Suppose that $(M,F)$ is a Finsler manifold. The pull-back bundle $\pi^\ast TM$ admits a unique torsion free linear
connection $\nabla$, called the Chern connection, which is almost g-compatible. The coefficients of the connection are of the form
\begin{equation*}
    \Gamma^i_{jk}=\frac{1}{2}g^{is}(\frac{\delta g_{sj}}{\delta x^k}-\frac{\delta g_{jk}}{\delta x^s}+\frac{\delta g_{ks}}{\delta x^j}),
\end{equation*}
where $\frac{\delta}{\delta x^j}:=\frac{\partial}{\partial x^j}-N^i_j\frac{\partial}{\partial y^i}$.
\end{theorem}

The concept of sectional curvature of Riemannian geometry is generalized to Finsler geometry which is called flag curvature.
The flag curvature is defined by
\begin{equation}\label{flag curvature main formula}
K(P,(x,y))=\dfrac{g_{(x,y)} \left( R_{(x,y)} \left( u\right) ,u\right) }{g_{(x,y)} (y,y) g_{(x,y)} (u,u)-g_{(x,y)}^2 (y,u)},
\end{equation}
where $P=\textit{span} \lbrace u,y \rbrace $ and $R$ is the Riemann curvature tensor defined by $R(u,v)w=\nabla _u \nabla _v w -\nabla _v \nabla _u w- \nabla _{[u,v]}w$ (for more details see \cite{Bao-Chern-Shen}).\\
Now for any $(x,y)\in TM\backslash \{0\}$ we define a linear transformation from $T_xM$ to itself by $R_{(x,y)}:=R(.,y)y$.

Then the Ricci curvature is defined as the trace of $R_{(x,y)}$. So we have
\begin{equation}\label{Ricci curvature}
    \mathcal{R}ic(x,y)=F^2(x,y)\sum_{i=1}^{n-1}K(P_i,(x,y)),
\end{equation}
where $P_i:=\textit{span} \lbrace e_i,y \rbrace$ and $\{e_1,\cdots,e_{n-1},\frac{y}{F(x,y)}\}$ is a $g_{(x,y)}$-orthonormal basis of $T_xM$.\\
Akbar-zadeh's Ricci tensor of a Finsler metric $F$ is defined as follows (see \cite{Akbar-Zadeh})
\begin{equation}\label{Akbar-zadeh's Ricci tensor}
    (ric_F)_{ij}:=\frac{1}{2}(\mathcal{R}ic(x,y))_{y_iy_j}.
\end{equation}

The following equation is considered by D. Bao, as the natural extension of un-normalised Ricci flow in Finsler geometry (see \cite{Bao}),
\begin{equation}\label{Ricci flow in Finsler geometry}
    \frac{\partial g_t}{\partial t}=-2 ric_{F_t}, \ \ \ \ F(t=0)=F_0.
\end{equation}
The concept of Ricci soliton is generalized to Finsler geometry by B. Bidabad and M. Yarahmadi as follows (see \cite{Bidabad-Yarahmadi}).\\
Suppose that $V$ is a vector field on the manifold $M$, such that $V=V^i\frac{\partial}{\partial x^i}$, for a local coordinate system. Then the complete lift of $V$ on $TM$, which is denoted by $V^c$, is a vector field on $TM$ locally defined by $V^c=V^i\frac{\partial}{\partial x^i}+y^j\frac{\partial V^i}{\partial x^j}\frac{\partial}{\partial y^i}$.

\begin{definition}\label{Finslerian Ricci soliton}
A Finsler manifold $(M,F)$ is said to be a Finslerian Ricci soliton if it satisfies the following equation,
\begin{equation}\label{Finslerian Ricci soliton equation}
    ric_F=cg+L_{V^c}g,
\end{equation}
for some vector field $V\in\mathfrak{X}(M)$, where $V^c$ denotes the complete lift of $V$ on $TM$ and $g$ is the Hessian related to $F$ (see \cite{Bidabad-Yarahmadi}).
\end{definition}
In theorem 3.4 of \cite{Bidabad-Yarahmadi}, the authors showed that if $(M,F_0)$ is a Finslerian Ricci soliton on a compact manifold $M$, then there exists a solution $(M,F_t)$ in the form
\begin{equation}\label{Ricci flow solution}
    F^2_t=\tau(t)\tilde{\varphi}^\ast_t F_0^2,
\end{equation}
to the Ricci flow equation \ref{Ricci flow in Finsler geometry}. Conversely, if $F_t$ of the form \ref{Ricci flow solution} is a solution to the Ricci flow equation \ref{Ricci flow in Finsler geometry}, then there exists a vector field $V$ on $M$ such that $(M,F_0)$ is a Finslerian Ricci soliton in the sense of definition \ref{Finslerian Ricci soliton}.

Here we mention that, similar to the Riemannian case, a Finsler metric $F$ on a Lie group $G$ is named left invariant if
\begin{equation}
F(x,y)=F(e,(l_{x^{-1}})_{\ast}y), \ \ \ \forall x\in G, y\in T_x G,
\end{equation}
where $e$ denotes the unit element of $G$.
%%---------------------------section 3--------------------------

\section{Main Results}
Here we generalize the concept of homogeneous Riemannian nilmanifold to Finsler spaces.
\begin{definition}
Let $(M,F)$ be a connected Finsler manifold with the Lie group of isometries $G=I_F(M)$. $M$ is said to be a homogeneous Finsler nilmanifold if $G$ contains a nilpotent Lie subgroup which acts transitively on $M$.
\end{definition}

\begin{remark}
Recall that for a Finsler manifold $(M,F)$, the group of isometries of $M$ is a Lie transformation group with respect to the compact-open topology (see \cite{Deng-Book} and \cite{Deng-Hou1}).
\end{remark}
The following proposition is a natural generalization of theorem $2$ of \cite{Wilson}.
\begin{prop}\label{isometry group-Finsler nilmanifolds}
Let $(M,F)$ be a Finsler manifold such that $I_F(M)\subset I_{\hat{\textbf{\textit{a}}}}(M)$, for a Riemannian metric ${\hat{\textbf{\textit{a}}}}$ on $M$, where $I_{\hat{\textbf{\textit{a}}}}(M)$ denotes the Lie group of isometries of $(M,{\hat{\textbf{\textit{a}}}})$. Suppose that $(M,F)$ is a homogeneous Finsler nilmanifold and $H$ is a nilpotent Lie subgroup of $I_F(M)$ acting transitively on $M$ with $N$ the connected component of the identity in $H$. Then
\begin{enumerate}
  \item $N$ acts simply transitively on $M$ and $N\unlhd G:=I_{\hat{\textbf{\textit{a}}}}(M)$,
  \item $\frak{n}=\frak{nil(g)}$, where $\frak{n}$ and $\frak{g}$ denote the Lie algebra of $N$ and $G$, respectively and $\frak{nil(g)}$ denotes the nilradical of $\frak{g}$,
  \item For any point $p_0\in M$, $G$ is the semi-direct product of $N$ and the isotropy subgroup $K_{p_0}=\{k\in G: kp_0=p_0\}$,
  \item $N$ is the unique nilpotent subgroup of $G$ acting simply transitively on $G$.
\end{enumerate}
\end{prop}

\begin{remark}
Let $F(x,y)=\sqrt{{\hat{\textbf{\textit{a}}}}(y,y)}+{\hat{\textbf{\textit{a}}}}(X_x,y)$ be a Randers metric on a Lie group $G$, arising from a left invariant Riemannian metric ${\hat{\textbf{\textit{a}}}}$ and a vector field $X$. In \cite{Salimi Moghaddam} (proposition 3.6) we have shown that, $F$ is left invariant if and only if $X$ is left invariant.
\end{remark}

\begin{prop}\label{isometry group of Randers spaces}
Let $F(x,y)=\sqrt{{\hat{\textbf{\textit{a}}}}(y,y)}+{\hat{\textbf{\textit{a}}}}(X_x,y)$ be a Randers metric on an arbitrary connected manifold $M$. Suppose that $I_F(M)$ and $I_{\hat{\textbf{\textit{a}}}}(M)$ denote the Lie groups of isometries of $(M,F)$ and $(M,{\hat{\textbf{\textit{a}}}})$, respectively. Then, $I_F(M)=I_{\hat{\textbf{\textit{a}}}}(M)$ if and only if $X$ is an $I_{\hat{\textbf{\textit{a}}}}(M)$-invariant vector field.
\end{prop}
\begin{proof}
Proposition 1.3 of \cite{Deng-Hou2} shows that, for an arbitrary Randers metric $F$, $I_F(M)\subseteq I_{\hat{\textbf{\textit{a}}}}(M)$. Suppose that $X$ is an $I_{\hat{\textbf{\textit{a}}}}(M)$-invariant vector field of $(M,{\hat{\textbf{\textit{a}}}})$ and $\phi\in I_{\hat{\textbf{\textit{a}}}}(M)$, so we have
\begin{eqnarray*}
  F(\phi(x),\phi_\ast(y)) &=& \sqrt{{\hat{\textbf{\textit{a}}}}(\phi_\ast(y),\phi_\ast(y))}+{\hat{\textbf{\textit{a}}}}(X_{\phi(x)},\phi_\ast(y)) \\
                          &=& \sqrt{{\hat{\textbf{\textit{a}}}}(\phi_\ast(y),\phi_\ast(y))}+{\hat{\textbf{\textit{a}}}}(\phi_\ast(X_x),\phi_\ast(y)) \\
                          &=& \sqrt{{\hat{\textbf{\textit{a}}}}(y,y)}+{\hat{\textbf{\textit{a}}}}(X_x,y)=F(x,y),
\end{eqnarray*}
which shows that $\phi\in I_F(M)$.\\
Conversely, suppose that $I_F(M)=I_{\hat{\textbf{\textit{a}}}}(M)$. If $\phi\in I_F(M)$ then
\begin{equation}\label{isometry eq1}
    F(x,y)=F(\phi(x),\phi_\ast(y))=\sqrt{{\hat{\textbf{\textit{a}}}}(\phi_\ast(y),\phi_\ast(y))}+{\hat{\textbf{\textit{a}}}}(X_{\phi(x)},\phi_\ast(y)).
\end{equation}
On the other hand $\phi\in I_{\hat{\textbf{\textit{a}}}}(M)$, so we have
\begin{equation}\label{isometry eq2}
    F(x,y)=\sqrt{{\hat{\textbf{\textit{a}}}}(\phi_\ast(y),\phi_\ast(y))}+{\hat{\textbf{\textit{a}}}}(\phi_\ast(X_x),\phi_\ast(y)).
\end{equation}
Hence for any $y\in T_xM$ we have
\begin{equation*}
    {\hat{\textbf{\textit{a}}}}(\phi_\ast(X_x)-X_{\phi(x)},\phi_\ast(y))=0.
\end{equation*}
So the proof is completed.
\end{proof}

\begin{cor}
Suppose that $F(x,y)=\sqrt{{\hat{\textbf{\textit{a}}}}(y,y)}+{\hat{\textbf{\textit{a}}}}(X_x,y)$ is a Randers metric on an arbitrary connected manifold $M$ such that $X$ is an $I_{\hat{\textbf{\textit{a}}}}(M)$-invariant vector field. Then $(M,F)$ is a homogeneous Finsler nilmanifold if and only if $(M,{\hat{\textbf{\textit{a}}}})$ is a homogeneous Riemannian nilmanifold.
\end{cor}

Proposition \ref{isometry group of Randers spaces} together with theorem 2 of \cite{Wilson} lead to the following corollary, which is very basic in this paper.

\begin{cor}\label{isometry group-Randers nilmanifolds}
Suppose that $(M,F)$ is a homogeneous Finsler nilmanifold such that $F$ is a Randers metric defined by $F(x,y)=\sqrt{{\hat{\textbf{\textit{a}}}}(y,y)}+{\hat{\textbf{\textit{a}}}}(X_x,y)$, where $X$ is $I_{\hat{\textbf{\textit{a}}}}(M)$-invariant. Suppose that $H$ is a nilpotent Lie subgroup of $I_F(M)$ acting transitively on $M$ with $N$ the connected component of the identity in $H$. Then
\begin{enumerate}
  \item $N$ acts simply transitively on $M$ and $N\unlhd G:=I_F(M)$,
  \item $\frak{n}=\frak{nil(g)}$,
  \item For any point $p_0\in M$, $G$ is the semi-direct product of $N$ and the isotropy subgroup $K_{p_0}$,
  \item $N$ is the unique nilpotent subgroup of $G$ acting simply transitively on $G$.
\end{enumerate}
\end{cor}
Now we use the above results to study Randers Ricci soliton.
\begin{definition}
A Finslerian Ricci soliton $(G,F)$ is named semialgebraic if it's $(1,1)$ Ricci tensor satisfies the equation,
\begin{equation}\label{Finslerian semialgebraic Ricci soliton equation}
    Ric_F=c\mathrm{Id}+\frac{1}{2}(\mathrm{D}+\mathrm{D}^t),
\end{equation}
for some $c\in\mathbb{R}$ and some $\mathrm{D}\in Der(\frak{g})$, where $G$ is a Lie group with Lie algebra $\frak{g}$, $F$ is a left invariant Finsler metric on $G$ and for any $y\in T_eG$, $\mathrm{D}^t$ denotes the transpose of $\mathrm{D}$ with respect to $g_{(e,y)}$.
\end{definition}

\begin{theorem}
Suppose that $F(x,y)=\sqrt{{\hat{\textbf{\textit{a}}}}_x(y,y)}+{\hat{\textbf{\textit{a}}}}_x(X_x,y)$ is a left invariant Randers metric on a simply connected nilpotent Lie group $N$, where ${\hat{\textbf{\textit{a}}}}$ is a left invariant Riemannian metric on $N$ and $X$ is an $I_{\hat{\textbf{\textit{a}}}}(M)$-invariant vector field. If the Ricci flow equation has a unique solution then, $(N,F)$ is a Ricci soliton if and only if $(N,F)$ is a semialgebraic Ricci soliton.
\end{theorem}
\begin{proof}
Let $F$ be a Ricci soliton. Then, there exists a one-parameter group of diffeomorphisms $\varphi_t$ on $N$ such that $g_t=\tau(t)\tilde{\varphi}^\ast_t g$ is the solution to the Ricci flow
\begin{equation}\label{Ricci flow equation}
    \frac{\partial g_t}{\partial t}=-2 ric_{F_t},
\end{equation}
with $g_0=g$, where $g_t$ denotes the fundamental tensor of $F_t$ and (see \cite{Bao} and \cite{Bidabad-Yarahmadi}).
The uniqueness of the solution shows that $\tilde{\varphi}^\ast_t g$ is also left invariant for any $t$.
Similar to the Riemannian case, fix a point $p\in N$ (see \cite{Jablonski2}). The smooth curve $\varphi_t(p)\subset N$ induces a smooth curve $h(t)\subset G=I_F(N)$ such that
\begin{equation*}
    l_{h(t)}.p=\varphi_t(p)\subset N\simeq G/G_p,
\end{equation*}
where $G_p$ denotes the isotropy at $p$. Let $\psi_t=l_{h(t)^{-1}}\circ\phi_t$. Since, for all $t$, $g_t=\tau(t)\tilde{\varphi}^\ast_t g$ is left invariant so $g_t=\tau(t)\tilde{\psi}^\ast_t g$ is a solution to the Ricci flow. Moreover, $\psi_t$ fixes the point $p$. Suppose that $\sigma\in I_F(N)=G$ is an arbitrary isometry, so we have,
\begin{equation*}
    \frac{\partial}{\partial t}\tilde{\sigma}^\ast g_t=\tilde{\sigma}^\ast\frac{\partial}{\partial t}g_t=-2\tilde{\sigma}^\ast ric_{F_t}=-2ric_{\tilde{\sigma}^\ast F_t}.
\end{equation*}
Therefore $\tilde{\sigma}^\ast g_t$ is a solution of the Ricci flow,
\begin{equation*}
    \frac{\partial}{\partial t}g_t=-2ric_{F_t},
\end{equation*}
with initial condition $\tilde{\sigma}^\ast g_0=g$. Again, the uniqueness of solutions shows that $\tilde{\sigma}^\ast g_t=g_t$, for all $t$. In other words $I_{F_t}(N)=I_F(N)=G$.\\
Let $\sigma\in I_F(N)=G$, so we have,
\begin{equation*}
    \psi_t\circ\sigma\circ\psi_t^{-1}\in I_{F_t}(N)=I_F(N)=G.
\end{equation*}
Hence, for any $\sigma\in G$ we have a smooth family
\begin{equation*}
    \Phi_t(\sigma)=\psi_t\circ\sigma\circ\psi_t^{-1}\in\mathrm{Aut}(G).
\end{equation*}
Easily we can see that $\Phi_t(G_p)=G_p$. Now corollary \ref{isometry group-Randers nilmanifolds} shows that $\psi_t=\phi_t$, where $\phi_t\in \mathrm{Diff}(N)$ is the diffeomorphism defined by $\Phi_t$ on $N$.\\
This means that there exists a one-parameter group $\psi_t$ of automorphisms of $N$ such that
$\tilde{\varphi}^\ast_t g=\tilde{\psi}^\ast_t g$, for all $t$. Now, let $\psi_t={\verb"e"}^{-t\mathrm{D}}$, where $\mathrm{D}\in Der(\frak{n})$. Then for any $(x,y)\in TN$ and $v,w\in\pi^{\ast}TN_{(x,y)}$ we have
\begin{equation*}
    \frac{\partial}{\partial t}|_{t=0}\tilde{\psi}^\ast_t g_{(x,y)}(v,w)=g_{(x,y)}(-(\mathrm{D}+\mathrm{D}^t)v,w).
\end{equation*}
On the other hand we have
\begin{equation*}
    \frac{\partial}{\partial t} g_t|_{t=0}=\tau'(0)g-g(\mathrm{D}+\mathrm{D}^t .,.).
\end{equation*}
Therefore
\begin{equation*}
    Ric_{F}=c\mathrm{Id}+\frac{1}{2}(\mathrm{D}+\mathrm{D}^t),
\end{equation*}
where $c=-\frac{\tau'(0)}{2}$.
\end{proof}

\begin{remark}
Easily we can see every (non-necessarily nilpotent) Finsler Lie group $(G,F)$, satisfying equation \ref{Finslerian semialgebraic Ricci soliton equation}, is a Finslerian Ricci soliton. It is sufficient to use the one-parameter group of automorphisms determined by $\mathrm{D}$.
\end{remark}

\end{document}